\documentclass{dcds}

\usepackage{amsmath}
\usepackage{amsfonts}
\usepackage{amsthm}
\usepackage{ifthen}
\usepackage{graphicx}

\newtheorem{theorem}{Theorem}
\newtheorem{statement}{Statement}
\newtheorem{lemma}{Lemma}
\newtheorem{assumption}{Assumption}
\newtheorem{proposition}{Proposition}

\theoremstyle{definition}
\newtheorem{definition}{Definition}

\DeclareMathOperator{\sgn}{sgn}
\DeclareMathOperator*{\argmin}{arg\,min}

\let\phi=\varphi
\def\eps{\varepsilon}
\def\id{\mathrm{id}}

\def\reals{\mathord{\mathbb{R}}}

\makeatletter
\def\abs#1{{\mathopen{\vert} #1 \mathclose{\vert}}}
\newcommand\Abs[2][DEFAULT]%
{%
  \ifthenelse{\equal{#1}{DEFAULT}}%
    {\left\vert #2 \right\vert}%
    {\@nameuse{#1l}\vert #2 \@nameuse{#1r}\vert}%
}
\makeatother

\title[Topological degree in analysis of canards]{Topological degree in analysis of canard-type trajectories in 3-D systems}

\thanks{This research was partially supported by SFI grant 05/RFP/ENG062
and by a private bequest to Cork University Foundation.
V.A. Sobolev was supported in part by the Russian Foundation for Basic
Research, grant 07-01-00169a, by Programme 22 of Presidium of RAS and
Programme 16 of the Branch OF Physical and Technical Problems of Energetics
of RAS}

\author[A.~Pokrovskii, D.~Rachinskii, V.~Sobolev and A.~Zhezherun]{}
\subjclass{Primary: ?????; Secondary: ?????}
\keywords{Keywords}

\email{a.pokrovskii@ucc.ie}
\email{d.rachinskii@ucc.ie}
\email{hsablem@yahoo.com}
\email{a.zhezherun@mars.ucc.ie}

\begin{document}

\maketitle

\centerline{\scshape Alexei Pokrovskii}
\medskip
{\footnotesize
 \centerline{Department of Applied Mathematics}
 \centerline{University College Cork, Ireland}
}
\medskip
\centerline{\scshape Dmitrii Rachinskii}
\medskip
{\footnotesize
 \centerline{Department of Applied Mathematics}
 \centerline{University College Cork, Ireland}
}
\medskip
\centerline{\scshape Vladimir Sobolev}
\medskip
{\footnotesize
 \centerline{Department of Differential Equations and Control Theory}
 \centerline{Samara State University, Russia}
}
\medskip
\centerline{\scshape Andrew Zhezherun}
\medskip
{\footnotesize
 \centerline{Department of Civil and Environmental Engineering}
 \centerline{University College Cork, Ireland}
}
\medskip
\centerline{(Communicated by )}
\medskip

\begin{abstract}
\end{abstract}

\section{Introduction}
Topological degree \cite{Deimling, geom} is one of the principal toolboxes of
the modern theory of nonlinear dynamical systems. The range of applications of
this toolbox is rapidly growing, see, for instance, \cite{Mawhin}. In this paper
we discuss a new, to the best of our knowledge, scheme of applying
topological degree to the analysis of canard-type trajectories.

If $W \colon \reals^d \mapsto \reals^d$ is a continuous mapping,
$\Omega \subset \reals^d$ is a bounded open set, and $y \in \reals^d$
does not belong to the image $W(\partial \Omega)$ of the boundary
$\partial \Omega$ of $\Omega$, then the symbol $\deg(W, \Omega, y)$
denotes the \emph{topological degree} \cite{Deimling} of $W$ at $y$
with respect to $\Omega$.
If $0 \not\in W(\partial \Omega)$, then the integer number
$\gamma(W, \Omega) = \deg(W, \Omega, 0)$, called the \emph{rotation of
the vector field $W$ at $\partial \Omega$}, is well defined.
A detailed description of properties of the number
$\gamma(W, \Omega)$ can be found, for example, in \cite{geom}.
In particular, if $\id$ denotes the identity mapping,
$\id(x) \equiv x$, then the number $\gamma(\id - W, \Omega)$
measures the algebraic number of fixed points of the mapping $W$
in $\Omega$.

We will use this tool to investigate canard-type periodic
trajectories of singularly perturbed differential equations. Let
us recall some related terminology. Consider the slow-fast system
\begin{alignat}{1}
\dot x &= f(x, y, z), \notag \\
\dot y &= g(x, y, z), \label{e:generic-3d} \\
\eps \dot z &= h(x, y, z), \notag
\end{alignat}
where $x$, $y$, $z$ are scalar functions of time, $\eps$ is a small
positive parameter, and $f$, $g$, $h$ are scalar functions. The subset
\[
S = \left\{ (x, y, z) \in \reals^3 \colon h(x, y, z) = 0 \right\}
\]
of the phase space is called a \emph{slow surface} of the system
\eqref{e:generic-3d}: on this surface the derivative $\dot z$ of the
fast variable is zero.
Moreover, a part of $S$ where
\[
\frac{\partial h}{\partial z}(x, y, z) < 0 \qquad
\left( \frac{\partial h}{\partial z}(x, y, z) > 0 \right)
\]
is called \emph{attractive} (\emph{repulsive}, respectively). A line
$L \subset S$ which separates attractive and repulsive parts of
$S$ will be called a \emph{turning line}. In what follows, we
suppose the turning line to be smooth. Trajectories which at first
pass along, and close to, an attractive part of $S$ and then continue
for a while along the repulsive part of $S$ are called \emph{canards} or
\emph{duck-trajectories} \cite{Benoit}.

In the paper, we focus our attention on periodic canards in singularly perturbed
picewise linear systems, or more specifically, in a special case of \eqref{e:generic-3d}
with $h(x, y, z) = x + \abs{z}$. This choice has a number of reasons.
From the methodological point of view, piecewise linear systems are convenient
because they are integrable. Furthermore, such systems are used extensively
in the modelling of a wide range of physical processes and have applications
in electrical circuits \cite{Kennedy,Bokhoven,Fujisawa}, flight control \cite{Porter},
chemical processes control \cite{Ozkan} and neural subsystems with
control behavior \cite{Mayeri}. The piecewise linearity may be due to
nonlinear elements such as saturation or may result from linearization about various
operating points of a nonlinear plant. For example, the McKean model is a piecewise
linear caricature of the FitzHugh-Nagumo model \cite{McKean}.

The main mathematical reason to consider the existence of periodic canards
for this class of differential systems is that they (together with classical canards)
are mathematical explanations of the limit behaviour of periodic solutions of
a family of differential equations. However, the traditional methods of the
``chasse au canard'' are adapted for sufficiently smooth systems, and we consider
the use of topological methods in the case of Lipschitzian nonlinearities as a dire necessity
(continuous piecewise linear functions satisfy the Lipschitz condition, see, for example,
\cite{Fujisawa}).

Note that system \eqref{e:generic-3d} has a vector slow variable $(x, y)$,
and it is known that the existence conditions for canards in systems with
scalar and vector slow variables have a vital difference. In the scalar case
it is necessary to have an additional parameter in the system under consideration,
which is demonstrated in \cite{Sekikawa,SIAM} where piecewise linear systems
on the plane are studied. The canard then exists only for a small range of
values of the parameter. In the vector case such a parameter is not required,
hence the system \eqref{e:generic-3d}, under some assumptions about the
right hand side given in the next section, always has a canard.
Below we show that \eqref{e:generic-3d} also has a periodic canard
if some additional assumptions are made.

\section{Main result} \label{s:main}

In this section we formulate the main existence result for
topologically stable canard type periodic trajectories, using
a simple example.

Consider a system of ordinary differential equations:
\begin{equation} \label{e:main}
{\arraycolsep=0pt \begin{array}{rl}
\dot x &{} = f(x, y, z), \\[3pt]
\dot y &{} = g(x, y, z), \\[3pt]
\eps \dot z &{} = x + \abs{z},
\end{array} }
\end{equation}
with the additional assumptions described below.

The slow surface of the system consists of two half-planes, an
attractive half-plane $P_a$ and a repulsive half-plane $P_r$:
\begin{alignat}{1}
P_a &= \{ (x, y, x) \colon x < 0 \}, \label{e:Pa} \\
P_r &= \{ (x, y, -x) \colon x < 0 \}, \label{e:Pr}
\end{alignat}
together with the turning line
\begin{equation} \label{e:line}
L = \{ (0, y, 0) \}.
\end{equation}

We consider auxiliary equations
\begin{equation} \label{e:attr}
{\arraycolsep=0pt \begin{array}{l}
\dot x = f_a (x, y) = f (x, y, x), \\[3pt]
\dot y = g_a (x, y) = g (x, y, x),
\end{array} }
\end{equation}
and
\begin{equation} \label{e:repul}
{\arraycolsep=0pt \begin{array}{l}
\dot x = f_r (x, y) = f (x, y, -x), \\[3pt]
\dot y = g_r (x, y) = g (x, y, -x),
\end{array} }
\end{equation}
which describe the dynamics near the slow half-planes
\eqref{e:Pa} and \eqref{e:Pr} in the limit $\eps \to 0$.

\begin{assumption} \label{a:f-g-bounds}
The functions $f$ and $g$ in the right-hand side of \eqref{e:main} are
globally bounded and globally Lipschitz continuous with a Lipschitz constant
$\lambda$:
\begin{equation} \label{e:fg-assump}
\abs{f(x, y, z)}, \abs{g(x, y, z)} < M.
\end{equation}
\end{assumption}

\begin{assumption} \label{a:f-g-in-zero}
The following relationships hold:
\[
f(0, 0, 0) = 0, \quad g(0, 0, 0) > 0, \quad y \cdot f(0, y, 0) < 0.
\]
\end{assumption}

An important role is played below by the solutions of \eqref{e:attr}
and \eqref{e:repul}. Denote by $w^*_a(t) = (x^*_a(t), y^*_a(t))$
the solution of the system \eqref{e:attr}, satisfying the initial
condition $x(0) = y(0) = 0$, and by $w^*_r(t) = (x^*_r, y^*_r)$ the solution
of the system \eqref{e:repul}, satisfying the same initial condition.

Assumption \ref{a:f-g-in-zero} ensures the existence of a finite time
interval $(T_a, T_r) \ni 0$ such that
\[
x^*_a(t) < 0 \mbox{ for } T_a < t < 0, \quad
x^*_r(t) < 0 \mbox{ for } 0 < t < T_r.
\]
Assumption \ref{a:f-g-in-zero} also implies strict limitation on the
possible location of canards of system \eqref{e:main}, which should
follow closely the attractive half-plane \eqref{e:Pa} for a
certain interval $t_a < t < 0$, and then move along the repulsive
half-plane \eqref{e:Pr} for $0 < t < t_r$. Any such canard
must first follow the curve
\begin{equation} \label{e:Ga}
\Gamma_a = \{ (x^*_a(t), y^*_a(t), x^*_a(t)) \} \subset P_a
\end{equation}
for negative times, and then the curve
\begin{equation} \label{e:Gr}
\Gamma_r = \{ (x^*_r(t), y^*_r(t), -x^*_r(t)) \} \subset P_r
\end{equation}
for positive times, passing near the origin, where the two curves
meet at $t = 0$. Using standard tools, see \cite{Arnold,Mishchenko},
it is easy to see that for any time interval $[t_a, t_r] \subset (T_a, T_r)$
with $t_a < 0 < t_r$ such a canard exists for every sufficiently
small $\eps > 0$.

The question whether there exist \emph{periodic} canards is less obvious. The
above argument shows that a periodic canard should have a segment of fast
motion from a small neighborhood of some point of the curve $\Gamma_r$ to a
small neighborhood of the curve $\Gamma_a$; this fast motion is, consequently,
almost vertical (i.e., almost parallel to the $z$ axis). More precisely, if
there is a limit of periodic canards as $\eps \to 0$, then the limiting closed
curve has necessarily a vertical segment connecting $\Gamma_r$ and $\Gamma_a$.
The next assumption ensures a possibility of such vertical jumps.

\begin{assumption} \label{a:intersect}
The trajectories $w^*_a(t)$ and $w^*_r(t)$ of systems
\eqref{e:attr} and \eqref{e:repul} intersect, that is, there
exist $\tau$ and $\sigma$ such that
\[
x^*_a(\tau) = x^*_r(\sigma) = x^*, \quad y^*_a(\tau) = y^*_r(\sigma) = y^*
\]
with
\[
T_a < \tau < 0 < \sigma < T_r.
\]
\end{assumption}

\begin{assumption} \label{a:transv}
The intersection is transversal, that is
\[
A = f_a(x^*, y^*) g_r(x^*,y^*) - f_r(x^*, y^*) g_a(x^*,y^*) \not= 0.
\]
\end{assumption}

Assumptions \ref{a:intersect}--\ref{a:transv} are illustrated
on Fig.\ \ref{fig:planes}.

\begin{figure}[htb]
\begin{center}
\includegraphics*[width=4.5cm]{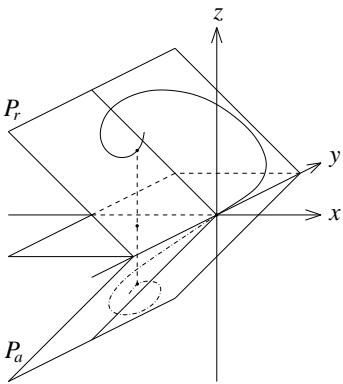}
\end{center}
\caption{Solutions $w^*_a(t)$ and $w^*_r(t)$.}
\label{fig:planes}
\end{figure}

The following theorem states the existence of periodic canard
in the system \eqref{e:main} under the Assumptions
\ref{a:f-g-bounds}--\ref{a:transv}.

\begin{theorem} \label{t:main}
For every sufficiently small $\eps > 0$ there exists a periodic solution
of system \eqref{e:main}. The minimal period $T_{\min}$ of
this solution approaches $\sigma - \tau$ as $\eps \to 0$.
\end{theorem}

It will also be shown that the periodic canard passes through a small
neighbourhood of the point $(x^*, y^*, 0)$, with the diameter of
this neighbourhood going to zero as $\eps$ goes to zero.

Below we discuss the main ideas behind the proof of this
theorem. The proof itself, which is divided into several
technical propositions and lemmas, will be provided in the next section.

Denote by
\[
w_a(t, t_0, x_0, y_0) = (x_a(t, t_0, x_0, y_0), y_a(t, t_0, x_0, y_0))
\]
the solution of system \eqref{e:attr} with the initial condition
\[
x(t_0) = x_0, \quad y(t_0) = y_0.
\]
Similarly, denote by
\[
w_r(t, t_0, x_0, y_0) = (x_r(t, t_0, x_0, y_0), y_r(t, t_0, x_0, y_0))
\]
the solution of system \eqref{e:repul} with the same initial condition
\[
x(t_0) = x_0, \quad y(t_0) = y_0.
\]

Consider a small vicinity of the point $(x^*, y^*)$. Since
the intersection between $\Gamma_a$ and $\Gamma_r$ is transversal
according to Assumption \ref{a:transv}, the numbers $u(x, y)$
and $v(x, y)$ may be defined in this vicinity by
\[
w_a(u(x, y), 0, x, y) \in \Gamma_r, \quad
w_r(v(x, y), 0, x, y) \in \Gamma_a.
\]

\begin{figure}[htb]
\begin{center}
\includegraphics*[width=5.5cm]{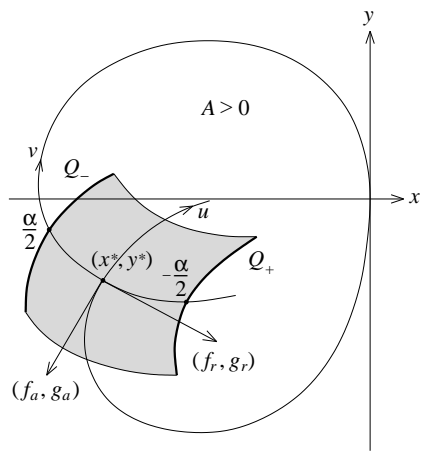}
\end{center}
\caption{The set $\Pi(\alpha)$ on the plane $(x, y)$.}
\label{fig:pi-rect}
\end{figure}

Using the coordinates $u(x, y)$ and $v(x, y)$ we introduce
for a sufficiently small $\alpha > 0$ a `parallelogram' set $\Pi(\alpha)$,
illustrated on Fig.\ \ref{fig:pi-rect}:
\begin{equation} \label{e:par}
\Pi(\alpha) = \{ x, y \colon \abs{u(x,y)}, \abs{v(x, y)} < \alpha / 2 \}.
\end{equation}
Also introduce the notation for the two `sides' of this parallelogram:
\begin{alignat*}{1}
Q_- &{} = \{ x, y \colon v(x, y) = +\alpha /2 \sgn A, \abs{u(x, y)} \le \alpha / 2 \}, \\
Q_+ &{} = \{ x, y \colon v(x, y) = -\alpha /2 \sgn A, \abs{u(x, y)} \le \alpha
/ 2 \}.
\end{alignat*}

Denote by
\begin{equation} \label{e:sol-eps}
w_\eps(t, x_0, y_0) = (x_\eps(t, x_0, y_0), y_\eps(t, x_0, y_0), z_\eps(t, x_0, y_0))
\end{equation}
the solution of the system \eqref{e:main} with the initial condition
\[
x(\tau) = x_0, \quad y(\tau) = y_0, \quad z(\tau) = 0,
\]
with $(x_0, y_0) \in \Pi(\alpha)$.

In our assumptions, the solution $w_\eps(t, x_0, y_0)$ first rapidly
approaches the attractive half-plane $P_a$ defined in \eqref{e:Pa},
and then follows $P_a$ until $t \approx 0$. The possible subsequent
behaviors of the solution are the following:

\begin{enumerate}
\item The solution stays close to the attractive part of the slow surface
for $t \le \delta$, where $\delta > 0$.

\item The solution begins to rapidly increase in the $z$ direction
around $t \approx 0$.

\item The solution follows the repulsive half-plane $P_r$ defined
in \eqref{e:Pr}, until a moment $t^*$. Moreover, for positive
$t < t^*$ it follows the curve $\Gamma_r$. After the moment $t^*$
the solution may begin to rapidly increase in the $z$ direction,
or it may fall down back to the attractive part of the slow surface.
\end{enumerate}

To distinguish between these cases, we define the time moment
\begin{equation} \label{e:s}
s_\eps(x_0, y_0) = \min ( \{ T_r \} \cup \{ t \colon
    \tilde{t}(x_0, y_0) + \rho \le t \le T_r \colon z_\eps(t, x_0, y_0) \le 0 \} ),
\end{equation}
where $\tilde{t} (x_0, y_0)$ is a time moment close to zero, which
is associated with the solution $w_a(t, \tau, x_0, y_0)$ and which
will be introduced in the next section, and $\rho > 0$ is a
sufficiently small number chosen together with $\alpha$.

Using the moment $s_\eps(x_0, y_0)$, we define an auxiliary mapping
$W_\eps$ of the parallelogram $\Pi(\alpha)$ into the plane $(x, y)$.
The definition is divided into the following four cases:

\smallskip
\noindent \textbf{Case 1.} If
\[
\sigma - \alpha < s_\eps(x_0, y_0) < \sigma + \alpha,
\]
then
\[
W_\eps(x_0, y_0) = (x_\eps(s_\eps, x_0, y_0), y_\eps(s_\eps, x_0, y_0)).
\]
If we identify the plane $(x, y)$ with the two-dimensional subspace
\[
P_0 = \{ (x, y, 0) \colon x, y \in \reals \}
\]
of the phase space of system \eqref{e:main}, then in Case~1 the
value $W_\eps(x_0, y_0)$ coincides with the intersection of the
trajectory \eqref{e:sol-eps} with $P_0$, as long as the
corresponding intersection time is close to $\sigma$.

\smallskip
\noindent \textbf{Case 2.} If
\[
\alpha \le \abs{s_\eps(x_0, y_0) - \sigma} < 2 \alpha,
\]
then
\[
W_\eps(x_0, y_0) = \frac{2 \alpha - \abs{s_\eps - \sigma}}{\alpha}
        (x_\eps(s_\eps, x_0, y_0), y_\eps(s_\eps, x_0, y_0))
    + \frac{\abs{s - \sigma} - \alpha}{\alpha} w^*_r(s_\eps).
\]
This means that $W_\eps(x_0, y_0)$ is a continuous convex combination of the
intersection of the trajectory \eqref{e:sol-eps} with $P_0$ and the point
$w^*_r(s_\eps)$ as long as the discrepancy $\abs{s_\eps - \sigma}$ between the
corresponding intersection moment $s_\eps$ and $\sigma$ in the the range from
$\alpha$ to $2 \alpha$. Moreover, if the discrepancy equals $\alpha$, then the
definition is consistent with Case~1; if the intersection moment equals
$\sigma \pm 2 \alpha$, then $W_\eps$ coincides with $w^*_r(\sigma \pm 2
\alpha)$.

\smallskip
\noindent \textbf{Case 3.} If
\[
s_\eps(x_0, y_0) \ge \sigma + 2 \alpha,
\]
then
\[
W_\eps(x_0, y_0) = w^*_r(\sigma + 2 \alpha).
\]

\smallskip
\noindent \textbf{Case 4.} If
\[
s_\eps(x_0, y_0) \le \sigma - 2 \alpha,
\]
then
\[
W_\eps(x_0, y_0) = w^*_r(\sigma - 2 \alpha).
\]

\smallskip
For small $\eps$, $\alpha$ and $\rho$ the time moment $s_\eps(x_0, y_0)$
depends continuously on $(x_0, y_0)$, therefore the operator
$W_\eps(x_0, y_0)$ is also continuous with respect to $(x_0, y_0)$.
We will show that if $(x_0, y_0)$ is a fixed point of $W_\eps$, then
Case 1 takes place, so the triple $(x_0, y_0, 0)$ defines a periodic
solution of \eqref{e:main} with the period $s_\eps - \tau$ close
to $\sigma - \tau$.

We also prove that the images of the parallelogram sides $Q_-$ and $Q_+$
are the points $w^*_r(\sigma + 2 \alpha)$ and $w^*_r(\sigma - 2 \alpha)$
correspondingly, and $u(W_\eps(x_0, y_0)) \to 0$ as $\eps \to 0$, thus
the operator $W_\eps$ is contracting along the $u$ coordinate, and
stretching along the $v$ coordinate. Finally, the vector field rotation
theory is used to obtain the relation
\[
\gamma(\id - W_\eps, \Pi(\alpha)) = \sgn A \not= 0,
\]
which implies the existence of a fixed point of the operator $W_\eps$
in the set $\Pi(\alpha)$, thus proving the existence of a periodic
canard.

\section{Proof of Theorem \ref{t:main}}

The proof consists of several parts. First, we formulate the
properties of the solutions $w_a(t, x_0, y_0) = w_a(t, \tau, x_0, y_0)$
of the system \eqref{e:attr} on the set $\Pi(\alpha)$. Secondly, we
consider the behavior of the solution $w_\eps(t, x_0, y_0)$
of the system \eqref{e:main} and establish the continuity of the
time moment $s_\eps(x_0, y_0)$ which was defined in \eqref{e:s} and
thus the continuity of the operator $W_\eps$ on $\Pi(\alpha)$.
Finally, we calculate the rotation of the vector field
$\id - W_\eps$ on $\Pi(\alpha)$.

We will need the following auxiliary definitions.

\begin{definition}
We say that $a$ is \emph{$\eps$-close} to $b$, if $a \to b$ as $\eps \to 0$.
$a$ is said to be \emph{$\eps$-small}, if it is $\eps$-close to 0.
\end{definition}

Consider a solution $w_a(t, x_0, y_0)$ of
the auxiliary system \eqref{e:attr} with the initial condition
$(x(\tau), y(\tau)) = (x_0, y_0) \in \Pi$.

Define the set
\[
R = \{ (x, y) \colon (x = 0 \wedge y \le 0) \vee (x \le 0 \wedge y = 0) \},
\]
which is the union of the left half of the horizontal $x$ coordinate axis and the
bottom half of the vertical $y$ coordinate axis. Consider all the time
moments ${\mathcal T}(x_0, y_0)$ when the solution $w_a(t, x_0, y_0)$
intersects the set $R$:
\[
{\mathcal T}(x_0, y_0) = \{ t \colon w_a(t, x_0, y_0) \in R \}.
\]

\begin{figure}[htb]
\begin{center}
\includegraphics*[width=3cm]{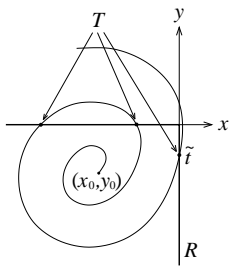}
\end{center}
\caption{The set ${\mathcal T}$ and the moment $\tilde t$.}
\label{fig:t-tilde}
\end{figure}

\begin{definition}
Let $\tilde t(x_0, y_0)$ be the moment from the set ${\mathcal T}(x_0, y_0)$
closest to zero, see Fig.\ \ref{fig:t-tilde}:
\[
\tilde t = \argmin_{t \in {\mathcal T}(x_0, y_0)} \abs{t}.
\]
Below we sometimes omit the point $(x_0, y_0)$ and write simply $\tilde{t}$,
in which case the arguments can be uniquely identified from the context.
\end{definition}

\begin{definition}
If $\tilde{t}(x_0, y_0)$ exists, and $x_a(\tilde t, x_0, y_0) = 0$
and $y_a(\tilde t, x_0, y_0) < 0$, then the solution $w_a(t, x_0, y_0)$
together with the point $(x_0, y_0)$ are called \emph{destabilizing}.
If $x_a(\tilde t, x_0, y_0) < 0$ and $y_a(\tilde t, x_0, y_0) = 0$, then
the solution $w_a(t, x_0, y_0)$ and the point $(x_0, y_0)$
are called \emph{stabilizing}.
\end{definition}

This definition emphasizes the fact that the solution
$w_\eps(t, x_0, y_0)$ of the main equation \eqref{e:main}
with a stabilizing initial condition will stay close
to the attractive half-plane $P_a$ for some time after $t > 0$,
if $\eps$ is sufficiently small; if initial condition
is destabilizing, then $z_\eps(t)$ will rapidly increase
to infinity after $t > 0$. The proof of this fact will be the
subject of several propositions, all leading to
Proposition \ref{p:s-cor}.

\begin{statement} \label{s:1}
If $\alpha$ is sufficiently small, then the moment $\tilde t$
is defined for all $(x_0, y_0) \in \Pi$ and is continuous
on $\Pi$ with respect to $(x_0, y_0)$.
\end{statement}

\begin{statement} \label{s:2}
If $\alpha$ is sufficiently small, then a point $(x_0, y_0) \in \Pi$ is
destabilizing if $A \cdot v(x_0, y_0) > 0$, and stabilizing if $A \cdot v(x_0,
y_0) < 0$. In particular, $Q_-$ is destabilizing and $Q_+$ is stabilizing.
\end{statement}

\begin{statement} \label{s:3}
Let $\alpha$ be sufficiently small. If $(x_0, y_0)$ is destabilizing,
then $x_a(t) < 0$ for $\tau \le t < \tilde t$, and if $(x_0, y_0)$ is
stabilizing, then $x_a(t) < 0$ for $\tau \le t \le \delta(\alpha)$
where $\delta(\alpha) > \max_\Pi \{ \tilde t (x_0, y_0) \} > 0$
depends only on $\alpha$.
\end{statement}

Statement \ref{s:1} follows from the continuous dependence of $w_a$ on $(x_0,
y_0)$ and from the fact that $w_a^*$ intersects the line $y = 0$
transversally.

To prove Statements \ref{s:2} and \ref{s:3}, we can consider the projection
$d(t)$ of the difference between the trajectories $w_a(t, x_0, y_0)$ and
$w_a^*(t)$ onto the normal vector for $w_a^*(t)$, $t \ge \tau$:
\[
d(t) = (x_a(t) - x_a^*(t)) g_a(x_a^*(t), y_a^*(t)) -
       (y_a(t) - y_a^*(t)) f_a(x_a^*(t), y_a^*(t)).
\]
Then the variation of $d(t)$ satisfies the following initial value problem:
\[
\dot r(t) = (\tfrac{\partial}{\partial x} f_a(x_a^*(t), y_a^*(t)) +
\tfrac{\partial}{\partial y} g_a(x_a^*(t), y_a^*(t))) r(t), \quad r(\tau) = 1,
\]
and the following equality holds for small $\Delta x_0 = x_0 - x_*$,
$\Delta y_0 = y_0 - y_*$, $d_0 = \Delta x_0 g_a(x_*, y_*) -
\Delta y_0 f_a(x_*, y_*)$:
\[
d(t) = r(t) d_0 + o(\Delta x_0) + o(\Delta y_0).
\]
Therefore, if $d_0 > 0$, and $\Delta x_0$ and $\Delta y_0$ are small,
then $d(t) > 0$ for $\tau \le t \le T_r$. Assumption \ref{a:transv}
provides that $d_0 > 0$ for $(x_0, y_0) \in Q_-$, and
$d_0 < 0$ for $(x_0, y_0) \in Q_+$.

Denote $\omega = \eps \ln \frac{T}{\eps}$ with $T = -T_a + T_r$,
$\omega > \eps$ for sufficiently small $\eps$.

Consider the solution $w_\eps(t, x_0, y_0)$ defined by \eqref{e:sol-eps}
with $(x_0, y_0) \in \Pi$.

\begin{proposition} \label{p:inv-sets}
The sets $z > -x + M \eps$ and $z > x - M \eps$ are invariant
in the sense that once a trajectory enters them, it does not leave them.
In particular, if $x_0 < 0$ and $z_0 = 0$, then
$z_\eps > x_\eps - M \eps$ for all $t \ge \tau$, and
if additionally $z_\eps(t_0) > -x_\eps + M \eps$ for some $t_0$,
then $z_\eps(t) > 0$ for all $t \ge t_0$.
\end{proposition}

\begin{proof}
Denote $\phi(t) = z(t) + x(t)$, $\psi(t) = z(t) - x(t)$.
Then \eqref{e:main} together with
assumption \eqref{e:fg-assump} imply that
\begin{alignat*}{1}
\dot \phi = {} & \dot z + \dot x = \frac{1}{\eps} (x + \abs{z}) + f(x, y, z)
    \ge \frac{1}{\eps} (x + z) - M = \frac{1}{\eps} \phi - M, \\
\dot \psi = {} & \dot z - \dot x = \frac{1}{\eps} (x + \abs{z}) - f(x, y, z)
    \ge \frac{1}{\eps} (x - z) - M = -\frac{1}{\eps} \psi - M,
\end{alignat*}
and by the theorem on differential inequalities,
\begin{gather*}
\phi(t) \ge \bar \phi(t) = (\phi_0 - M \eps) e^{\frac{t - t_0}{\eps}} + M \eps, \\
\psi(t) \ge \bar \psi(t) = (\psi_0 + M \eps) e^{-\frac{t - t_0}{\eps}} - M \eps,
\end{gather*}
with $\phi_0 = \phi(t_0) = z(t_0) + x(t_0)$. and
$\psi_0 = \psi(t_0) = z(t_0) - x(t_0)$. If $\phi_0 > M \eps$, then
$\bar\phi(t) > M \eps$ and thus $z(t) > -x(t) + M \eps$ for all $t \ge t_0$.
Similarly, if $\psi_0 > - M \eps$, then $\bar\psi(t) > -M \eps$, therefore
$z(t) > x(t) - M \eps$ for all $t \ge t_0$.

The second part of the proposition follows from the fact that
the point $(x_0, y_0, 0)$ with $x_0 < 0$ is in the set $z > x - M \eps$, and
if both $z > x - M \eps$ and $z > -x + M \eps$ hold, then $z > 0$.
\end{proof}

\begin{proposition} \label{p:z-x-attr}
Let $\eps$ be sufficiently small. Denote
\begin{equation} \label{e:hat-t}
\hat t = \min( \{ T_r \} \cup \{ t > \tau \colon z_\eps(t) = 0 \}),
\end{equation}
and
\begin{equation} \label{e:t1}
t_1 = \tau + \eps \ln \Bigl( -\frac{x_0}{M \eps} \Bigr).
\end{equation}
Then $t_1 < \tau + \omega$, $\hat t > t_1$, and for $t_1 \le t \le \hat t$:
\[
\abs{z_\eps(t) - x_\eps(t)} < 2 M \eps.
\]
\end{proposition}

\begin{proof}
According to Proposition \ref{p:inv-sets}, the solution $w_\eps(t)$ does not leave the set
$z - x > -M \eps$, because $(x_0, y_0) \in \Pi$ and thus $z_0 - x_0 > 0$. Therefore,
it suffices to prove the relation $z_\eps(t) - x_\eps(t) < 2 M \eps$ for
$t_1 \le t \le \hat t$.

First we prove that $z_\eps(t) < 0$ for $\tau < t \le t_1$, which will imply
$\hat t > t_1$.
Note that $-M T < x_0 < -M \eps$, thus $t_1 < \tau + \omega$.
Assuming that $z(t) < 0$, we get the following relations for $\psi = z - x$:
\begin{gather}
\dot \psi = \dot z - \dot x = \frac{1}{\eps} (x - z) - f(x, y, z), \notag \\
\dot \psi \ge -\frac{1}{\eps} \psi - M, \quad
    \psi_0 = z_0 - x_0 > M \eps, \label{e:psi-below} \\
\dot \psi(t) \ge (\psi_0 + M \eps) e^{-\frac{1}{\eps} (t - \tau)} - M \eps
    = \psi^{(1)}(t). \notag
\end{gather}
The function $\psi^{(1)}(t)$ is decreasing, and
\[
\psi^{(1)} (t_1) = (\psi_0 + M \eps) \frac{M \eps}{\psi_0} - M \eps > 0,
\]
therefore $\dot z_\eps = -\frac{1}{\eps} \psi \le -\frac{1}{\eps} \psi^{(1)} < 0$
for $\tau \le t \le t_1$, which proves that $z_\eps(t) < 0$ for $\tau < t \le t_1$.

Now, similarly to \eqref{e:psi-below}, we write
\begin{gather*}
\dot \psi \le -\frac{1}{\eps} \psi + M, \\
\dot \psi(t) \le (\psi_0 - M \eps) e^{-\frac{1}{\eps} (t - \tau)} + M \eps = \psi^{(2)}(t), \\
\psi^{(2)} (t_1) = (\psi_0 - M \eps) \frac{M \eps}{\psi_0} + M \eps < 2 M \eps,
\end{gather*}
thus for $t \ge t_1$:
\[
z_\eps(t) - x_\eps(t) \le \psi^{(2)} (t) < 2 M \eps.
\]
This inequality holds while $z_\eps(t)$ satisfies the equation
$\dot z = \frac{1}{\eps} (x - z)$, i.e. while $t \le \hat t$.
\end{proof}

As before, $(x_a, y_a)$ will denote the solution of the system \eqref{e:attr}
with the initial condition $x(\tau) = x_0$, $y(\tau) = y_0$.

\begin{proposition} \label{p:x-y-attr}
If $\eps$ is sufficiently small, then the following inequalities hold
for $\tau \le t \le \hat t$:
\[
\abs{x_\eps(t) - x_a(t)}, \abs{y_\eps(t) - y_a(t)} \le L \omega,
\]
where $L = 5 M e^{4 \lambda T}$.
\end{proposition}

\begin{proof}
The moment $t_1$ defined by \eqref{e:t1} in Proposition \ref{p:z-x-attr}
is $\omega$-close to $\tau$, thus
\[
\abs{x_\eps(t_1) - x_a(t_1)} \le \abs{x_\eps(t_1) - x_0} + \abs{x_a(t_1) - x_0}
    \le 2 M \omega, \quad
\abs{y_\eps(t_1) - y_a(t_1)} \le 2 M \omega.
\]
Denote $\psi(t) = \abs{x_\eps(t) - x_a(t)} + \abs{y_\eps(t) - y_a(t)}$, $t \ge t_1$.
Using \eqref{e:main} and \eqref{e:attr} together with Assumption \ref{a:f-g-bounds}
and Proposition \ref{p:z-x-attr}, we obtain
\begin{gather*}
D_R \abs{x_\eps - x_a} \le \abs{f(x_\eps, y_\eps, x_\eps)
    - f(x_a, y_a, x_a)} + 2 \lambda M \eps \le
    2 \lambda \abs{x_\eps - x_a} + \lambda \abs{y_\eps - y_a} + 2 \lambda M \eps, \\
D_R \abs{y_\eps - y_a} \le \abs{g(x_\eps, y_\eps, x_\eps)
    - g(x_a, y_a, x_a)} + 2 \lambda M \eps \le
    2 \lambda \abs{x_\eps - x_a} + \lambda \abs{y_\eps - y_a} + 2 \lambda M \eps,
\end{gather*}
where $D_R$ denotes the right derivative. Therefore
\[
D_R \psi(t) \le 4 \lambda \psi(t) + 4 \lambda M \eps, \quad \psi(t_1) \le 4 M \omega.
\]
and by the theorem on differential inequalities,
\[
\psi(t) \le (4 M \omega + M \eps) e^{4 \lambda (t - t_1)} - M \eps
    \le L \omega.
\]
for $t_1 \le t \le \hat t$.
\end{proof}

\begin{proposition} \label{p:unstable}
If $(x_0, y_0)$ is destabilizing and $\eps$ is sufficiently small, then the
solution $(x_\eps, y_\eps, z_\eps)$ reaches $z = 0$ at the moment $\hat t$
which is $\eps$-close to $\tilde t(x_0, y_0)$, and $z_\eps(t) > 0$
for all $t > \hat t$.
\end{proposition}

\begin{proof}
Denote $\tilde y = y(\tilde t) < 0$, $\tilde f = f (0, \tilde y, 0) > 0$,
$\tilde g = g(0, \tilde y, 0)$. Consider a small vicinity $\Omega$ of the point
$(0, \tilde y, 0)$:
\[
\Omega = \{ (x, y, z) \colon \abs{x}, \abs{z}, \abs{y - \tilde y} \le \delta \}
\]
Due to the continuity of $f$ and $g$ we can select a sufficiently small
$\delta$ such that $\tilde f / 2 < f(x, y, z) < 2 \tilde f$,
and $\abs{g(x, y, z)} < 2 \max \{ \tilde f, \abs{\tilde g} \}$
for all $(x, y, z) \in \Omega$. Denote
\[
\Delta \tilde t = \delta / ( 2 \max \{ \tilde f, \abs{\tilde g} \} ),
\quad \tilde x = x_a(\tilde t - \Delta \tilde t) < 0.
\]
Then $(x_a(t), y_a(t), x_a(t)) \in \Omega$ for
$\tilde t - \Delta \tilde t \le t \le \tilde t + \Delta \tilde t$,
$x_a(t)$ is increasing on this interval, and Statement \ref{s:3}
implies that
\begin{equation} \label{e:x-a-neg}
x_a(t) \le \tilde x \mbox { for } \tau \le t \le \tilde t - \Delta \tilde t.
\end{equation}

First we will prove that $\hat t$ is $\eps$-close to $\tilde t$.
Let $\Delta t = 4 L \omega / \tilde f < \Delta \tilde t$ for
sufficiently small $\eps$. We have
\begin{equation} \label{e:x-a-bounds}
x_a(\tilde t - \Delta t) < - 2 L \omega, \quad x_a(\tilde t + \Delta t) > 2 L \omega.
\end{equation}
Consider the interval $I = \{ t \colon t_1 \le t \le \tilde t - \Delta t \}$.
Inequalities \eqref{e:x-a-neg}, \eqref{e:x-a-bounds} together with the fact
that $x_a(t)$ is increasing for $\tilde t - \Delta \tilde t \le t \le \tilde t - \Delta t$
imply that $x_a(t) < -2 L \omega$ for $t \in I$. Applying Propositions
\ref{p:z-x-attr} and \ref{p:x-y-attr}, we obtain
\[
z_\eps(t) < x_\eps(t) + 2 M \eps < x_a(t) + L \omega + 2 M \eps < 0
    \mbox{ for } t \in I,
\]
therefore $z_\eps(t) < 0$ for $\tau < t \le \tilde t - \Delta t$.

Now suppose that $z_\eps(t) < 0$ for $\tilde t - \Delta t \le t \le \tilde t + \Delta t$.
Then Propositions \ref{p:z-x-attr} and \ref{p:x-y-attr} will hold on this interval,
and \eqref{e:x-a-bounds} will imply that
\[
z_\eps(\tilde t + \Delta t) > x_\eps(\tilde t + \Delta t) - 2 M \eps
> x_a(\tilde t + \Delta t) - L \omega - 2 M \eps > 0.
\]
This contradiction proves that $\hat t$ lies in the interval
$\tilde t - \Delta t \le \hat t \le \tilde t + \Delta t$.

Finally, we need to prove that $z_\eps(t) > 0$ for $t > \hat t$.
We have
\[
\dot z_\eps(\hat t) = \frac{1}{\eps} ( x_\eps(\hat t)
    + z_\eps(\hat t) ) \ge 0,
\]
thus $x_\eps(\hat t) \ge 0$. It will suffice to show that
the solution enters the set $z > -x + M \eps$ before the time
$\hat t + O (\eps)$ (so that the solution will still be in the
set $\Omega$ by that time, and $x_\eps(t)$ will be increasing,
which will guarantee that $z_\eps(t) \ge 0$).
Let $\phi(t) = z_\eps(t) + x_\eps(t)$, then
\[
\dot \phi(t) = \frac{1}{\eps} \phi(t) + f(x_\eps, y_\eps, z_\eps)
    \ge \frac{1}{\eps} \phi(t) + \frac{1}{2} \tilde f,
\quad \phi(\hat t) \ge 0,
\]
and by the theorem on differential inequalities,
\[
\phi(t) \ge \frac{1}{2} \tilde f \eps \left( e^{\frac{1}{\eps} (t - \hat t)} - 1 \right)
    \mbox{ for } t \ge \hat t.
\]
Consider the time moment $\hat t + \eps \ln \frac{4 M}{\tilde f}$.
At this moment
\[
\phi \left( \hat t + \eps \ln \frac{4 M}{\tilde f} \right)
    \ge \frac{1}{2} \tilde f \eps \left( \frac{4 M}{\tilde f} - 1 \right)
    > \frac{1}{2} (4 M - \tilde f) \eps > M \eps.
\]
The proposition is proved.
\end{proof}

\begin{proposition} \label{p:stable}
If $(x_0, y_0)$ is stabilizing and $\eps$ is sufficiently small,
then $z_\eps < 0$ for all $\tau < t \le \delta(\alpha)$.
\end{proposition}

\begin{proof}
According to Statement \ref{s:3}, $x_a(t) < 0$ for $\tau \le t \le \delta(\alpha)$.
Denote
\[
\tilde x = \min_{\tau \le t \le \delta(\alpha)} x_a(t) < 0.
\]
If $\eps$ is sufficiently small, then $x_a(t) < -2 M \eps$ on this
interval, thus $z_\eps(t) < x_\eps(t) + 2 M \eps < 0$ for all
$t_1 \le t \le \delta(\alpha)$.
\end{proof}

Denote
\[
s_\eps(x_0, y_0) = \min (\{ T_r \} \cup
    \{ t \ge \tilde t(x_0, y_0) + \rho \colon z_\eps(t) \le 0 \} ).
\]
The parameter $\rho$ is sufficiently small and is chosen in the following way:
due to Assumption \ref{a:f-g-in-zero}, there exists a $\delta$-vicinity of
$(0, 0, 0)$ such that $f(x, y, z)$ is small and $g(x, y, z)$ is positive. Then
$\rho$ must be less than $\delta / (2 M)$, so that if a solution $(x_\eps,
y_\eps, z_\eps)$ is $\eps$-small at the moment $\hat t$, then it
remains in this vicinity for $\hat t - 2 \rho < t < \hat t + 2 \rho$, and
$y_\eps$ increases on this interval. Additionally, $\rho$ should satisfy the
inequality $\rho < \alpha / 4$.

\begin{proposition} \label{p:s-cor}
If $(x_0, y_0)$ is destabilizing, then for sufficiently small $\eps$,
$s = T_r$. If $(x_0, y_0)$ is stabilizing, then for sufficiently small $\eps$,
$s = \rho$.
\end{proposition}

\begin{proof}
Follows from Propositions \ref{p:unstable} and \ref{p:stable}.
\end{proof}

\begin{proposition} \label{p:hat-t}
Let $z_\eps(s) \ge 0$ and $\eps$ be sufficiently small, and let $\hat t$
be given by \eqref{e:hat-t}. Then $\abs{\hat t - \tilde t(x_0, y_0)} \le \rho / 4$,
$0 \le \abs{x_\eps(\hat t)} \le M \eps$, and $y_\eps(\hat t)$ is $\eps$-small.
\end{proposition}

\begin{proof}
From the definition of $\hat t$ it follows that
$\dot{z}_\eps(\hat t) = \frac{1}{\eps}(x_\eps(\hat t) + \abs{z_\eps(\hat t)}) \ge 0$,
thus $x_\eps(\hat t) \ge 0$, and the fact that $z_\eps(s) = 0$
implies that the solution never enters the set $z + x > M \eps$,
which is invariant according to Proposition \ref{p:inv-sets}.
Therefore $x_\eps(\hat t) \le M \eps$.

Denote $y_\eps(\hat t) = \hat y$. Using the same reasoning
as in Proposition \ref{p:unstable}, it can be shown that
whatever $\hat y < 0$ is, for sufficiently small values
of $\eps$ the solution $z_\eps(t) > 0$ for $t > \hat t$.
Similarly, if $\hat y > 0$, then $f(x, y, z) < 0$ in some
vicinity of the point $(0, \hat y, 0)$, and for sufficiently
small $\eps$ there exists $\Delta t$ such that
$x_\eps(\hat t - \Delta t) > 2 M \eps$, $z_\eps(\hat t - \Delta t) < 0$,
which contradicts Proposition \ref{p:z-x-attr}.

Therefore, there exists a function $\gamma(\eps) \to 0$ as $\eps \to 0$,
such that $\abs{\hat y} < \gamma(\eps)$.

It remains to prove that $\abs{\hat t - \tilde t(x_0, y_0)} \le \rho / 4$.
First we show that $\hat t \ge \tilde t(x_0, y_0) - \rho / 4$.
This follows from the fact that $x_a(t) < 0$ for
$\tau \le t \le \tilde t - \rho / 4$, so for sufficiently small
$\eps$, $x_\eps(t) < x_a(t) + L \eps < 0$ for $t \le \tilde t - \rho / 4$.

To prove the relation $\hat t \le \tilde t + \rho / 4$, we note that
Proposition \ref{p:s-cor} implies
\[
    \abs{x_a(\tilde t)}, \abs{y_a(\tilde t)} < \gamma(\eps).
\]
Suppose that $\hat t > \tilde t + \rho / 4$. Then, taking into account
Propositions \ref{p:z-x-attr} and \ref{p:x-y-attr},
we get that both $y_\eps(\tilde t)$ and $y_\eps(\hat t)$ are
$\eps$-small, which is impossible because in this case $y_\eps$
should increase for $\tilde t \le t \le \hat t$ due to the choice
of $\rho$. Therefore, $\hat t \le \tilde t + \rho / 4$
and the Proposition is proved.
\end{proof}

\begin{proposition} \label{p:bar-t}
Let $z_\eps(s) \ge 0$ with a sufficiently small $\eps$.
Denote
\[
\bar t = \min \{ t \le s \colon z_\eps(\xi) \ge 0 \mbox{ for } t \le \xi \le s \}.
\]
Then $\abs{\bar t - \tilde t(x_0, y_0)} \le \rho / 2$,
$0 \le \abs{x_\eps(\bar t)} \le M \eps$, and
$y_\eps(\bar t)$ is $\eps$-small.
\end{proposition}

\begin{proof}
By the definition of $s$ and $\bar t$, $\hat t \le \bar t \le \tilde t +
\rho$. According to Proposition \ref{p:hat-t}, the solution is $\eps$-small
at the moment $\hat t$. Thus, due to the choice of $\rho$, the function
$y_\eps(t)$ increases on the interval $\hat t \le t \le \bar t$. Therefore,
$y_\eps(\bar t) > y_\eps(\hat t) > -\gamma(\eps)$ with an $\eps$-small
$\gamma(\eps)$.

Consider the solution $(x_\eps, y_\eps, z_\eps)$ in backward time
at the moment $\bar t$. Using the same technique as in Propositions
\ref{p:unstable} and \ref{p:hat-t}, it can be shown that there exists
an $\eps$-small function $\gamma(\eps)$ such that
$y_\eps(\bar t) < \gamma(\eps)$.

It remains to show that $\bar t \le \tilde t + \rho / 2$. This follows
from the fact that $\hat t \le \tilde t + \rho / 4$, thus the
function $y_\eps$ increases for $\hat t < t < \bar t$, therefore the
$\eps$-smallness of both $y_\eps(\hat t)$ and $y_\eps(\bar t)$
implies that $\hat t$ and $\bar t$ are also $\eps$-close.
\end{proof}

\begin{proposition} \label{p:repul}
Let $z_\eps(s) \ge 0$ with a sufficiently small $\eps$. Then
\begin{equation} \label{e:z-x-repul}
\abs{z_\eps(t) + x_\eps(t)} \le 2 M \eps
\end{equation}
for $\bar t \le t \le s - \omega$, and
\begin{equation} \label{e:x-y-repul}
\abs{x_\eps(t) - x^*_r(t - \bar t)}, \abs{y_\eps(t) - y^*_r(t - \bar t)} \le L \gamma,
\end{equation}
for $\bar t \le t \le s$, where $\gamma$ is $\eps$-small.
\end{proposition}

\begin{proof}
To prove the formula \eqref{e:z-x-repul}, consider the solution
$(x_\eps, y_\eps, z_\eps)$ in backward time with the initial
condition $(x_\eps(s), y_\eps(s), z_\eps(s))$
at the moment $\hat s$. The the proof follows that of the of
Proposition \ref{p:z-x-attr}, with the exception that the fact
that $z_\eps(t) \ge 0$ for $s - \omega \le t < s$
is now provided by the statement of the Proposition
($s - \bar t \ge \rho / 2$, thus the interval above is non-empty).

The proof of the formula \eqref{e:x-y-repul} uses \eqref{e:z-x-repul} and
Proposition \ref{p:bar-t}, and is similar to that of Proposition
\ref{p:x-y-attr}.
\end{proof}

\begin{proposition} \label{p:s-cont}
If $\eps$ is sufficiently small, then the moment of time $s_\eps(x_0, y_0)$ is
continuous on $\Pi$ with respect to $x_0$ and $y_0$.
\end{proposition}

\begin{proof}
Let $(x_0, y_0) \in \Pi$ and prove that $s$ is continuous at the point
$(x_0, y_0)$. There are three possible cases:

\begin{enumerate}

\item $z_\eps(s) < 0$, $s = \tilde t + \rho$. Then, due to continuity of
$z_\eps$ and $\tilde t$, $z_\eps(\tilde t + \rho; \hat x_0, \hat y_0) < 0$
for $(\hat x_0, \hat y_0)$ close to $(x_0, y_0)$, therefore
$s(\hat x_0, \hat y_0) = \tilde t(\hat x_0, \hat y_0) + \rho$, and
$s$ is continuous at $(x_0, y_0)$.

\item $z_\eps(s) > 0$, $s = T_r$. Then $z_\eps(t; \hat x_0, \hat y_0) > 0$
for $\tilde t(\hat x_0, \hat y_0) + \rho \le t \le T_r$,
thus $s(\hat x_0, \hat y_0) = T_r$.

\item $z_\eps(s) = 0$. According to Proposition \ref{p:repul}, $x_\eps(s)$
is $\eps$-close to $x_r^*(t - \bar t) < 0$, thus $x_\eps(s) < 0$ and
$\dot z_\eps(s) = \frac{1}{\eps}(x_\eps(s) + \abs{z_\eps(s)}) < 0$. Therefore
$z_\eps(t)$ intersects the plane $z = 0$ transversally.
Then, due to continuity of $z_\eps$, there exists a moment $\hat s$
close to $s$, such that $z_\eps(\hat s; \hat x_0, \hat y_0) = 0$. If
$\hat s < \tilde t(\hat x_0, \hat y_0) + \rho$, then $z_\eps$ decreases for
$\hat s \le t \le \tilde t(\hat x_0, \hat y_0) + \rho$, thus
$z_\eps(\tilde t + \rho) < 0$.
\end{enumerate}

In all cases, $s$ is continuous.
\end{proof}

\begin{proposition} \label{p:u-v-small}
Let $z_\eps(s_\eps, x_0, y_0) = 0$ with a sufficiently small $\eps$.
Then both $v(x_0, y_0)$ and $u(x_\eps(s_\eps), y_\eps(s_\eps))$
are $\eps$-small.
\end{proposition}

\begin{proof}
According to Proposition \ref{p:hat-t}, $w_\eps(\hat t)$ is $\eps$-small.
Then by following the solutions $w_\eps(t)$ and $w_a^*(t)$ in backward time
we obtain that $w_\eps(\tau)$ is $\eps$-close to $w_a^*(\hat t - \tau)$, which
implies that $v(x_0, y_0)$ is $\eps$-small. The second part of this Proposition
is obtained in the same way by following the solutions $w_\eps(t)$ and $w_r^*(t)$
in forward time, as in Proposition \ref{p:repul}.
\end{proof}

\begin{lemma} \label{l:w-continuous}
If $\eps$ is sufficiently small, then operator $W_\eps(x_0, y_0)$ is
continuous with respect to $(x_0, y_0)$.
\end{lemma}

\begin{proof}
Follows from the definition of $W_\eps$, see section \ref{s:main},
and from Proposition \ref{p:s-cont}.
\end{proof}

\begin{lemma} \label{l:fixed-point}
Let $(\hat x, \hat y) \in \Pi(\alpha)$ be a fixed point of the operator
$W_\eps(x_0, y_0)$. Then the solution $w_\eps(t; \hat x, \hat y)$
is periodic.
\end{lemma}

\begin{proof}
If suffices to show that if $(\hat x, \hat y)$ is a fixed point, then Case 1
from the definition of $W_\eps$ holds, so that the value of $W_\eps$
corresponds to a point on the trajectory $w_\eps(t; \hat x, \hat y)$ and is
not adjusted, as in Cases 2--4. Note that in Cases 3 and 4, $W_\eps(\hat x,
\hat y) \not\in \Pi$, thus they cannot hold. It remains to prove that Case 2
also cannot hold.

Suppose that Case 2 holds, and let, for example, $s_\eps(\hat x, \hat y) \ge
\sigma + \alpha$. Denote $(\bar x, \bar y) = w_\eps(s_\eps, \hat x, \hat y)$.
According to Proposition \ref{p:u-v-small}, both $v(\hat x, \hat y)$
and $u(\bar x, \bar y)$ are $\eps$-small. Moreover, according to
Propositions \ref{p:hat-t} and \ref{p:bar-t},
\[
\alpha \le s_\eps(\hat x, \hat y) - \sigma < u(\hat x, \hat y) - v(\bar x,
\bar y) + \rho.
\]
Remember that $\abs{u(\hat x, \hat y)} \le \alpha / 2$, and
$\rho < \alpha / 4$. Thus,
\[
v(\bar x, \bar y) < - \alpha / 4.
\]
Note that the linear combination in the definition of $W_\eps$ moves
the point $(\bar x, \bar y)$ in the direction of the point
$(0, - 2 \alpha)$ in $(u, v)$-coordinates, thus further
decreasing $v(\bar x, \bar y)$. Therefore, we obtain
$v(\hat x, \hat y) < -\alpha / 4$, which contradicts Proposition
\ref{p:u-v-small}. This proves that only Case 1 can hold for
$(\hat x, \hat y)$.
\end{proof}

\begin{lemma} \label{l:rotation}
For sufficiently small $\eps$ the rotation $\gamma(\id - W_\eps, \Pi(\alpha))$
of the vector field $p - W_\eps(p)$ at the boundary of the set $\Pi(\alpha)$
is defined by
\[
\gamma(\id - W_\eps, \Pi(\alpha)) = \sgn(A).
\]
\end{lemma}

\begin{proof}
Let for example,
\[
A > 0.
\]
Consider $Q_-$ and $Q_+$, the upper and lower sides of
the parallelogram $\Pi(\alpha)$. Then, by definition,
\begin{alignat*}{2}
W_\eps(x, y) &= w^*_r(\sigma + 2\alpha), &\quad (x, y) &\in Q_{-}, \quad \text{and} \\
W_\eps(x, y) &= w^*_r(\sigma - 2\alpha), &\quad (x, y) &\in Q_{+}.
\end{alignat*}
In other words,
\begin{alignat}{2}
u(W_\eps(x, y)) &= 2\alpha,  &\quad (x, y) &\in Q_{-}, \label{e:f1} \\
u(W_\eps(x, y)) &= -2\alpha, &\quad (x, y) &\in Q_{+}. \label{e:f2}
\end{alignat}
Also, according to Proposition \ref{p:bar-t},
\begin{equation} \label{e:f3}
\lim_{\eps \to 0} v(W_\eps(x, y)) = 0.
\end{equation}
The relationships \eqref{e:f1}--\eqref{e:f3} imply that in the coordinates
$(u, v)$ the mapping $W_\eps$ on the boundary of $\Pi$ is close to the linear
mapping $L_1(u, v) = (0, -4v)$, thus the mapping $\id - W_\eps$ is close to
(and therefore co-directed with) the linear mapping $L_2(u, v) = (u, 5v)$, and
the result follows from the properties of the rotation number.
\end{proof}

Lemma \ref{l:rotation} implies that the operator $W_\eps(x_0, y_0)$ has a
fixed point $(\hat x, \hat y)$ on the set $\Pi(\alpha)$, which defines a periodic
solution of \eqref{e:main} according to Lemma \ref{l:fixed-point}.
Moreover, according to Propositions \ref{p:s-cor} and \ref{p:bar-t},
$v(\hat x, \hat y)$ and $u(\hat x, \hat y)$ are both $\eps$-small,
which implies both that $(\hat x, \hat y)$ is $\eps$-close to $(x^*, y^*)$,
and that the minimal period of the solution starting from $(\hat x, \hat y)$
is $\eps$-close to $\sigma - \tau$. Thus, Theorem \ref{t:main} is proved.

\section{Example}
Consider the system
\begin{equation*}
{\arraycolsep=0pt \begin{array}{rl}
\dot x &{} = - a y + z / a, \\[3pt]
\dot y &{} = x + 1, \\[3pt]
\eps \dot z &{} = x + \abs{z}.
\end{array} }
\end{equation*}
In this case \eqref{e:attr} takes the form
\begin{equation*}
{\arraycolsep=0pt \begin{array}{rl}
\dot x &{} = - a y + x / a, \\[3pt]
\dot y &{} = x + 1,
\end{array} }
\end{equation*}
and \eqref{e:repul} turns into
\begin{equation*}
{\arraycolsep=0pt \begin{array}{rl}
\dot x &{} = - a y - x / a, \\[3pt]
\dot y &{} = x + 1.
\end{array} }
\end{equation*}

For $a > a_0 \approx 1.9$ the corresponding curves $w_a(t)$ and $w_r(t)$
intersect transversally on the plane $(x, y)$, see Fig.\ \ref{fig:sample-2d}.

\begin{figure}[htb]
\begin{center}
\begin{tabular}{cc}
\includegraphics*[width=5cm]{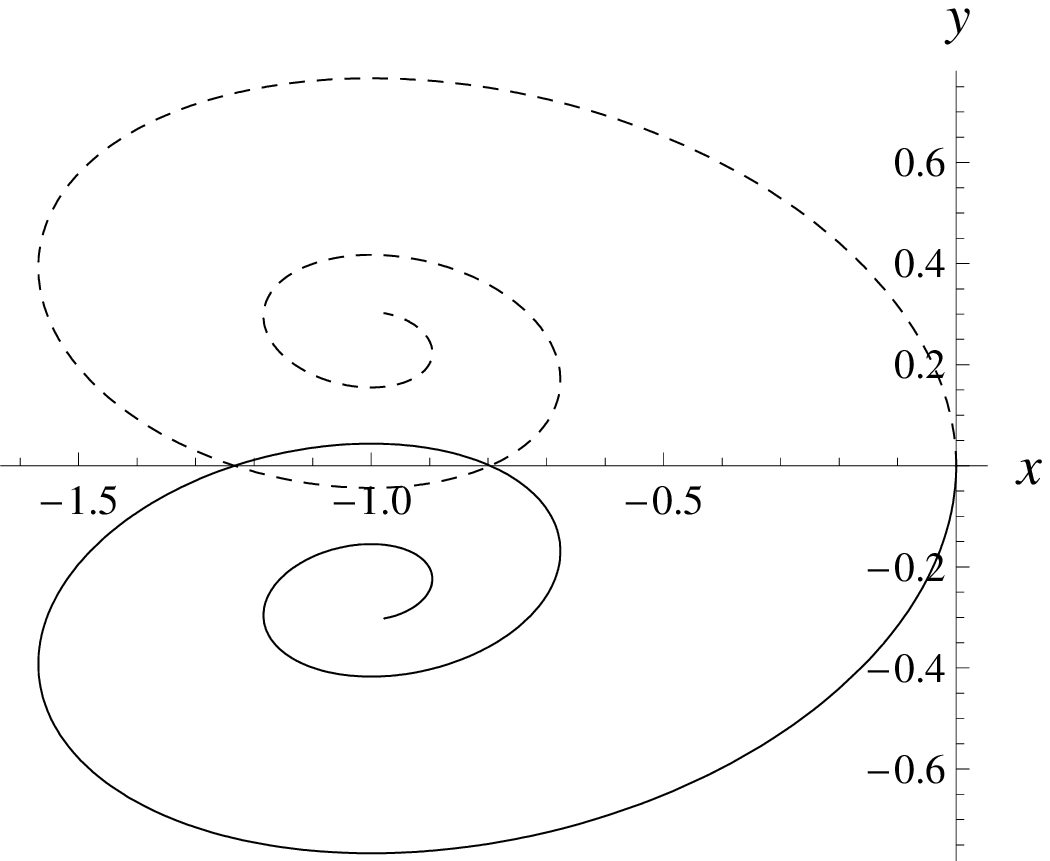} &\quad
\includegraphics*[width=5cm]{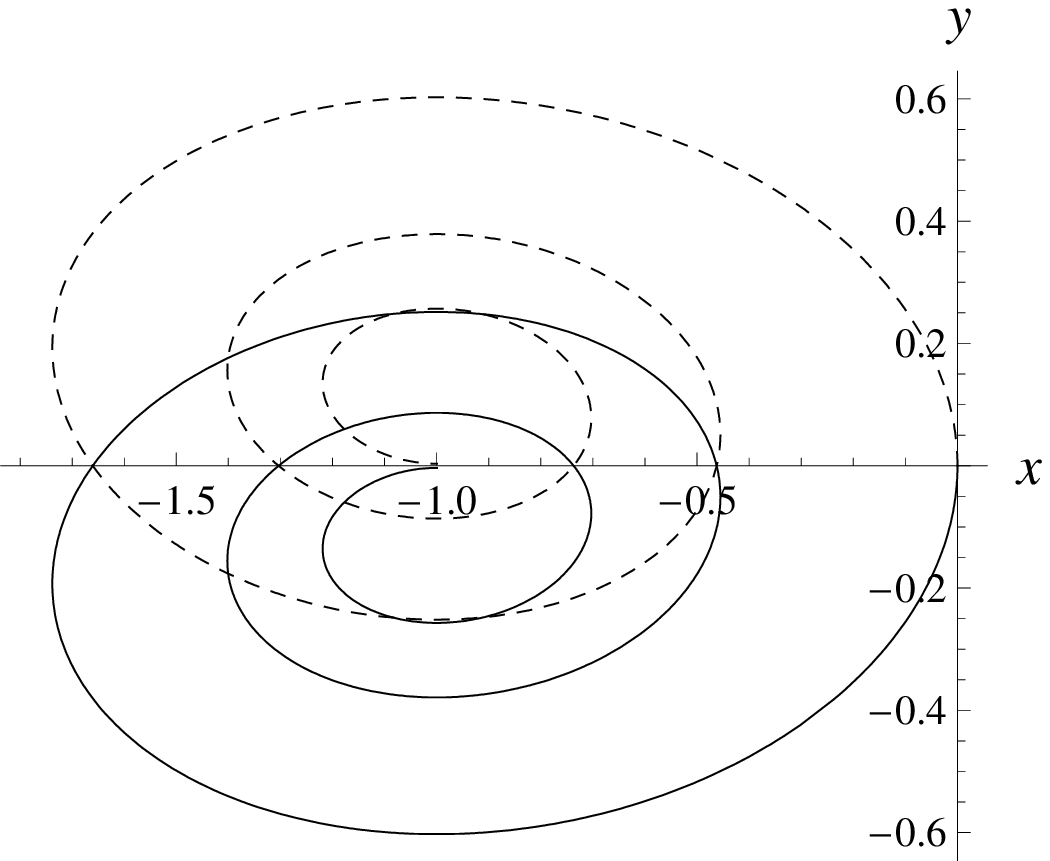}
\end{tabular}
\end{center}
\caption{Trajectories $w_a(t)$ (solid) and $w_r(t)$ (dashed); left: $a = 2$, right: $a = 3$.}
\label{fig:sample-2d}
\end{figure}

For any $a > a_0$ this system has a periodic canard. The last figure
graphs the numerical approximation of such trajectory, together
with the limiting curve, which consists of $\Gamma_a$, $\Gamma_r$,
and a vertical segment connecting them.

\begin{figure}[htb]
\begin{center}
\includegraphics*[width=5.5cm]{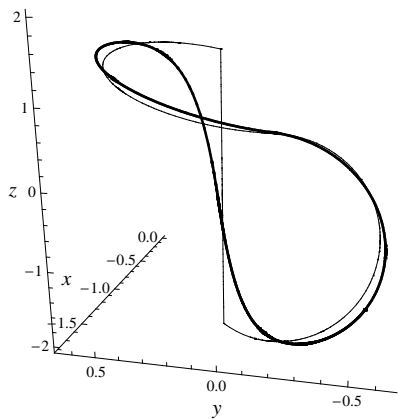}
\end{center}
\caption{Periodic canard for $a = 3$ with $\eps = 0.1$ (thick) and the limiting curve.}
\label{fig:sample-canard}
\end{figure}

\end{document}